\documentclass{article}
\usepackage{graphicx} 
\usepackage{amsmath}
\usepackage{amssymb}
\usepackage{amsthm}
\usepackage[numbers,sort&compress]{natbib}
\usepackage{url}

\newcommand{\la}{\left|}
\newcommand{\ra}{\right|}
\newcommand{\lb}{\left\{}
\newcommand{\rb}{\right\}}
\newcommand{\lc}{\left(}
\newcommand{\rc}{\right)}
\newcommand{\ld}{\left[}
\newcommand{\rd}{\right]}
\newcommand{\li}{\langle}
\newcommand{\ri}{\rangle}
\newcommand{\norm}[1]{\left\lVert#1\right\lVert}
\newcommand{\h}{\frac{1}{2}}
\newcommand{\infI}{\int_{-\infty}^{\infty}}
\newcommand{\hI}{\int_{-\h}^{\h}}

\newtheorem{thm}{Theorem}
\newtheorem{lem}{Lemma}
\newtheorem{prop}{Proposition}

\theoremstyle{definition}
\newtheorem{remark}{Remark}

\title{Unconditional Bounds for Two, Three and Four Consecutive Moderate Gaps between Zeros of the Riemann Zeta Function}
\author{}
\date{October 2026}

\begin{document}

\author{
  C\'esar Ang\'on\thanks{Email:\texttt{ceangon@ucsb.edu}}
  \and
  Erick Ross\thanks{Deceased.}
  \and
  Ze Sen Tang\thanks{Email:\texttt{zxt230006@utdallas.edu}}
  \and
  Boyang Zhao\thanks{Corresponding author.
  Email:\texttt{bzhao23@buffalo.edu}.}
}

\maketitle
\begin{center}
\emph{In memory of Erick Ross, whose insight, generosity, and love of mathematics inspired this work.}
\end{center}
\begin{abstract}
We prove unconditional positive-proportion results for consecutive moderate
gaps between zeros of the Riemann zeta function. More precisely, for every
fixed \(0<c\leq 1/4\), the lower asymptotic densities, normalized by
\(N(T)\), of occurrences of two, three, and four consecutive gaps,
each of length at least \(2\pi c/\log T\), are at least
\(0.3450014\ldots\), \(0.1812517\ldots\), and \(0.0175021\ldots\),
respectively. Here zeros are ordered by ordinate and counted with
multiplicity, and \(N(T)\) denotes the number of nontrivial zeros with
ordinates in \((0,T]\). These bounds address the corresponding cases of a question of Gonek and Sahay. The proof combines the methods of Alp\"oge--Furman and Lamzouri with local spectral estimates and cluster counting.
\end{abstract}

\section{Introduction}

Let \(0<\gamma_1\leq \gamma_2\leq \cdots\) denote the positive
ordinates of the nontrivial zeros of the Riemann zeta function,
counted with multiplicity, and let \(N(T)\) denote the number of
such zeros with ordinates in \((0,T]\). By the Riemann--von Mangoldt
formula, the mean spacing between consecutive ordinates near height
\(T\) is asymptotic to
\[
\frac{2\pi}{\log T}.
\]

Following Gonek and Sahay~\cite{GS24}, for fixed \(c>0\) we call
a gap \(\gamma_{n+1}-\gamma_n\) \(c\)-moderate if
\[
\gamma_{n+1}-\gamma_n\geq \frac{2\pi c}{\log T}.
\]
For an integer \(\ell\geq 2\), let \(w_\ell(T,c)\) denote the number
of indices \(1\leq n\leq N(T)-\ell\) for which the \(\ell\)
consecutive gaps
\[
\gamma_{n+1}-\gamma_n,\ \ldots,\
\gamma_{n+\ell}-\gamma_{n+\ell-1}
\]
are all \(c\)-moderate. The corresponding sets
\(\mathcal W_\ell(T,c)\) are defined precisely in Section~2.

The problem concerns the joint local behavior of neighboring gaps.
A large difference \(\gamma_{n+\ell}-\gamma_n\) does not imply that
each of the \(\ell\) intervening gaps is moderate, since one or more
of them may still be very small. Thus, a lower bound for the total
span alone does not guarantee a run of consecutive moderate gaps.

Gonek and Sahay~\cite[Theorem~2.1 and Corollary~2.2]{GS24} proved, assuming the Riemann hypothesis and the pair coorelation conjecture, that such runs have positive lower density for every fixed\(\ell\), provided\(c > 0\) is sufficiently small.
They also noted that, at the time of their work, even the existence of infinitely many pairs of consecutive moderate gaps was not known unconditionally for and fixed \(c>0\).

Our main result gives uniform positive-proportion lower bounds for
two, three and four consecutive moderate gaps, without assuming
either the Riemann hypothesis or the pair correlation conjecture.

\begin{thm}\label{THM}
For every fixed \(0<c\leq \frac14\) and every integer
\(2\leq \ell\leq 4\),
\[
\liminf_{T\to\infty}\frac{w_\ell(T,c)}{N(T)}
\geq
1-\frac{\ell+2}{2}\bigl(C_{\mathrm{MT}}-1\bigr)
=
\frac{\ell+4-(\ell+2)C_{\mathrm{MT}}}{2},
\]
where
\[
C_{\mathrm{MT}}
=
\frac12+\frac{1}{\sqrt2}\cot\!\left(\frac{1}{\sqrt2}\right)
=
1.3274992963205\ldots.
\]
\end{thm}

In particular, the lower bounds for two, three and four consecutive
moderate gaps are
\[
0.3450014\ldots,\qquad
0.1812517\ldots,\qquad
0.0175021\ldots,
\]
respectively. The positivity of the four-gap bound follows from
\(C_{\mathrm{MT}}<\frac43\). The constant \(C_{\mathrm{MT}}\) comes
from the classical Montgomery--Taylor extremal window; see
Montgomery~\cite{Mon75} and \cite[Corollary~14]{CCLM17}.

\begin{remark}
After the first version of this paper appeared on arXiv, Youness Lamzouri posted his preprint~\cite{Lam26Isolated} and informed us that he had been working independently on the same problem. He obtained unconditional positive-proportion results for two, three, four, and five consecutive moderate gaps. For two, three and four gaps, his stated results give higher proportions for smaller values of \(c\), whereas our bounds remain valid at the larger gap threshold \(c = \frac14\).
\end{remark}

Our argument is based on the unconditional finite-dimensional method
of Alp{\"o}ge and Furman~\cite{AF26} and the Hilbert-space formulation
of Lamzouri~\cite{Lam26}. The analytic input is Lamzouri's unweighted
second-moment estimate~\cite[Lemma~3.2]{Lam26}, based on the
unconditional pair-correlation formula of Baluyot, Goldston,
Suriajaya and Turnage-Butterbaugh~\cite[Lemma~5]{BGSTB24}.

We group nearby simple critical line zeros into blocks and partition each block into pairs and triples. Local spectral estimates give a uniform additional contribution per point in each block. This yields a bound for the combined contribution of the two obstructions in terms of the second moment.
A separate counting argument converts that bound into a lower bound for runs in the full sequence of ordinates, accounting for multiple zeros and distinct zeros sharing an ordinate.
Together, these estimates give Theorem~\ref{THM}

Sections~2--5 introduce the second-moment estimate and the matrix framework. Section~6 proves the local spectral estimates, Section~7 gives the counting argument for consecutive gaps, and Section~8 completes the proof of Theorem~\ref{THM}.

\section{Preliminaries}

\subsection{Notations for Zeros}
Set $\rho = \beta + i \gamma$ to be the nontrivial zeros of the Riemann zeta function, and
\begin{equation}
    N(T) = \sum_{\substack{\rho \\ 0 < \gamma \leq T}} 1
\end{equation}
be the number of nontrivial zeros from $0$ to $T>0$. Now, we define a multi-set
\begin{equation}
    Z_T = \lb i\lc\rho - \h \rc \frac{\log T}{2\pi} : 0 < \gamma\leq T \rb,
\end{equation}
where $\la Z_T \ra = N(T)$ and $\overline{Z_T}=Z_T$. Note that the critical zeros will be a real number in $Z_T$. Now, we split $Z_T$ into three disjoint multi-sets
\begin{equation}
    R_1=\lb x_1,...,x_n \rb, \quad R_2 = \lb x_{n+1},...,x_{n+1},x_{n+2},...,x_{n+2},...,x_{n+r},...,x_{n+r} \rb ,
\end{equation}
\begin{equation*}
    R_3 = \lb z_1,\overline{z}_1, ... ,z_1,\overline{z}_1,z_2,\overline{z}_2, ... ,z_2,\overline{z}_2,..., z_k,\overline{z}_k, ...,  z_k,\overline{z}_k \rb,
\end{equation*}
where $R_1$ is a set of simple critical zeros, $R_2$ is of multiple critical zeros, $R_3$ is of pairs off-line zeros, and $m_z$ is the multiplicity of $z \in Z_T$. Moreover, we have
\begin{equation}
    \#\lb \text{distinct zeros in }Z_T \rb = n+r+2k.
\end{equation}
Let $m_z$ be the multiplicity of $z \in Z_T$, then we have
\begin{equation}
    N(T) = \sum_{l=1}^{n+r} m_{x_l} +\sum_{l=1}^k 2m_{z_l}.
\end{equation}
Finally, let \[0<\gamma_1 \le \gamma_2\leq \cdots \le \gamma_{N(T)} \le T\] be the ordinates of the nontrivial zeros up to height \(T\), counted with multiplicity.
For \(c>0\) and an integer \(\ell \ge 2\), define \[\mathcal W_\ell(T,c) = \left\{
j\in\mathbb Z:1 \le j \le N(T) - \ell, \quad
\gamma_{j+a+1} - \gamma_{j+a}\ge \frac{2\pi c}{\log T}\  \forall \quad 0\le a < \ell \right\}\]
and put
\[w_\ell(T,c) = \# \mathcal W_\ell(T,c).\]
Thus, \(w_\ell(T,c)\) counts the runs of \(\ell\) consecutive \(c\)-moderate gaps between zeros with ordinates in \((0,T]\)

\subsection{Notations for Test Function}
For a compactly supported function \(f\in L^1(\mathbb R)\), we set
\begin{equation}
    \hat{f}(\xi) = \infI f(u) e^{-2\pi i \xi u} du \quad \text{for } \xi \in \mathbb{C}
\end{equation}
as the Fourier transform of $f$.
We use the following second-moment estimate from
Lamzouri~\cite[Lemma~3.2]{Lam26}.
\begin{lem} \label{Youness}
    Let 
    \begin{equation}
        C_{MT} = \h + \frac{1}{\sqrt2} \cot{\frac{1}{\sqrt{2}}} = 1.3274992963205....
    \end{equation}
        There exists a family of real-valued even functions \(\eta_\delta \in C_c^\infty((-\frac12, \frac12))\), \(0< \delta <\frac14\), such that, on setting \(K_\delta(z):=\widehat{\eta_\delta^2}(z)\), we have \(K_\delta(0) = 1\) and, for every fixed \(\delta\),

    \begin{equation} \label{Youness eq}
        \sum_{\substack{\rho,\rho' \\ 0<\gamma,\gamma'\leq T}} K_{\delta}^2 \lc i\lc \rho-\rho' \rc \frac{\log T}{2\pi} \rc = \lc C_{\eta_{\delta}}+O_{\delta}\lc \frac{1}{\sqrt{\log T}}\rc \rc N(T),
    \end{equation}
where \(C_{\eta_\delta}\) is defined below and satisfies
\[
C_{\eta_\delta} \longrightarrow C_{\mathrm{MT}}
\qquad \text{as } \delta \to 0.
\]
\end{lem}
The estimate in \cite[Lemma~3.2]{Lam26} is originally normalized
by $T\log T/(2\pi)$. By the Riemann--von Mangoldt asymptotic
$N(T)=T\log T/(2\pi)+O(T)$~\cite[Theorem~9.4]{Titchmarsh1986},
replacing that normalization by
$N(T)$ introduces a relative error of $O(1/\log T)$, which is
absorbed by the displayed error term.

Following the smoothing construction in
\cite[proof of Lemma~3.2]{Lam26}, we construct $\eta_{\delta}$ as follows.
Set $I = \ld -\h,\h \rd$ and define
\begin{equation}
    A(x) = \begin{cases}
        \frac{\cos{\sqrt{2}x}}{\sqrt{2}\sin{\frac{1}{\sqrt{2}}}}, &\text{if }x\in I,\\
        0, &\text{else}.
    \end{cases}
\end{equation}
So, $A$ is real, even, strictly positive on $I$ and $\infI A(x)dx = 1$. 

Let $\delta \in \left(0,\frac{1}{4}\right)$ and $\widetilde{\eta}_{\delta} \in C_c^{\infty}\left( \left(-\h,\h\right) \right)$ be an even function such that $0\leq \widetilde{\eta}_{\delta}\leq 1$ and $\widetilde{\eta}_{\delta}(x)=1$ for $|x|\leq \h-\delta$. Then, we can define
\begin{equation}
    A_{\delta} = \infI \widetilde{\eta}_{\delta}^2(x)A(x) dx
\end{equation}
and set
\begin{equation} \label{eta}
    \eta_{\delta} = \frac{\widetilde{\eta}_{\delta}(x)\sqrt{A(x)}}{\sqrt{A_\delta}}.
\end{equation}
Then, we have $\eta_{\delta},\eta^2_{\delta} \in C_c^{\infty}\left( \left(-\h,\h\right) \right)$ are real and even. Moreover, $\eta^2_{\delta}\geq 0$ and $\infI \eta^2_{\delta}(x)dx=1$.\\
All these functions are extended by zero outside \(I\).
For each \(\delta\), fix one such cutoff independently of \(T\),
and set
\[
K_\delta(z):=\widehat{\eta_\delta^2}(z).
\]
Then \(K_\delta(0)=1\).

For every fixed \(0<\delta<1/4\), the argument in the proof of
\cite[Lemma~3.2]{Lam26}, with the normalization explained above,
gives
\begin{equation}\label{eq:delta-second-moment}
\sum_{\substack{\rho,\rho'\\0<\gamma,\gamma'\leq T}}
K_\delta\!\left(
i(\rho-\rho')\frac{\log T}{2\pi}
\right)^2
=
\left(
C_{\eta_\delta}
+
O_\delta\!\left(\frac{1}{\sqrt{\log T}}\right)
\right)N(T),
\end{equation}
where
\[
C_{\eta_\delta}
:=
(\eta_\delta^2*\eta_\delta^2)(0)
+
2\int_0^1
\alpha(\eta_\delta^2*\eta_\delta^2)(\alpha)\,d\alpha,
\]
and \(*\) denotes convolution. The same argument shows that
\[
C_{\eta_\delta}\longrightarrow C_{\mathrm{MT}}
\qquad\text{as }\delta\to0.
\]

\section{Vector Space of $f_z,g_z,h_z$}
We follow the real Hilbert-space setup in the proof of
\cite[Proposition~2.1]{Lam26} and record its coordinate realization.

First, we assign each element $z \in Z_T$ to an element in $L^2\lc \lc -\h,\h \rc \rc$. For each $z \in Z_T$, we define
\begin{equation}
    f_z(u) = \eta_{\delta}(u)e^{-2\pi i uz}.
\end{equation}
Then, for $z,s \in Z_T$, if $z \neq s$, then we have $f_z$ and $f_{s}$ are linearly independent. Moreover, the inner product between $f_z$ and $f_{s}$ is defined as
\begin{equation}
    \li f_z,f_s \ri = \hI f_z(u)\overline{f_s(u)} du = \hI \eta_{\delta}^2(u)e^{-2\pi i (z-\overline{s})u}du = \hat{\eta_{\delta}^2}(z-\overline{s}) = K_{\delta}(z-\overline{s}).
\end{equation}
Moreover, we have
\begin{equation}
    \sum_{z,s \in Z_T} K_{\delta}(z-\overline{s}) = \sum_{z,s \in Z_T} K_{\delta}(z-s)
\end{equation}
since $\overline{Z_T} = Z_T$. Note that if $z$ is real, we must have
\begin{equation}\label{one1}
    ||f_z||^2 = \li f_z,f_z \ri = \hI \eta_{\delta}^2(u)du = K_\delta(0) = 1.
\end{equation}
Let
\begin{equation}
    g_z(u) = \frac{f_z(u)+f_{\overline{z}}(u)}{2} \quad \text{and} \quad h_z(u) = \frac{f_z(u)-f_{\overline{z}}(u)}{2i}.
\end{equation}
Then, for any $z \in Z_T$, we have
\begin{equation}\label{K}
\begin{aligned}
    1 &= K_\delta(0) = K_\delta(z-z) =\hI f_z(u)\overline{f_{\overline{z}}(u)} du \\
      &= \hI \la g_z(u) \ra^2 - \la h_z(u) \ra^2 + 2i \mathrm{Re}(g_z(u)\overline{h_z(u)})du.
\end{aligned}
\end{equation}
Since the LHS of \eqref{K} is real, the imaginary part of RHS must be zero. Thus, we have
\begin{equation}\label{one2}
    \hI \la g_z(u) \ra^2 du  - \hI \la h_z(u) \ra^2 du =1.
\end{equation}
Observe that
\begin{equation} \label{basis}
    f_{x_1}, ... , f_{x_{n+r}},g_{z_1}, ... g_{z_k},h_{z_1},...,h_{z_k}
\end{equation}
are linearly independent. Then we can define a vector space by setting \eqref{basis} as the basis and get
\begin{equation}
    W = span\lb f_{x_1}, ... , f_{x_{n+r}},g_{z_1}, ... g_{z_k},h_{z_1},...,h_{z_k} \rb,
\end{equation}
where $\dim(W)=n+r+2k$.

Now, observe that all $f_x,g_z,h_z$ satisfy the property $\overline{\phi(u)} = \phi(-u)$. Then we have
\begin{equation} \label{real}
\begin{aligned}
    \overline{\li \phi_1,\phi_2 \ri} = \hI \overline{\phi_1(u)}\phi_2(u)du &= \hI \phi_1(-u)\overline{\phi_2(-u)}du\\
    &= \hI \phi_1(u)\overline{\phi_2(u)}du = \li \phi_1,\phi_2 \ri.
\end{aligned}
\end{equation}
This implies the inner product must be real. Now we can apply Gram-Schmidt process to construct an orthonormal basis for $W$, denoted as $\lb \psi_1, ... , \psi_{\dim(W)} \rb$, which also satisfies \eqref{real}.

\section{Construction of $G = P + Q$}
We combine the finite-dimensional viewpoint of Alp{\"o}ge and
Furman~\cite[Section~3]{AF26} with Lamzouri's multiset
Hilbert-space framework~\cite[proof of Proposition~2.1]{Lam26}.
The identities and spectral estimates used below are proved explicitly.

For each $\phi_z \in \lb f_{x_1}, ... , f_{x_{n+r}},g_{z_1}, ... g_{z_k},h_{z_1},...,h_{z_k} \rb$, we define
\begin{equation}
    v_{\phi_{z}} = \lc \li \phi_z , \psi_1 \ri, ... ,\li \phi_z , \psi_{\dim(W)} \ri \rc \in \mathbb{R}^{\dim(W)}.
\end{equation}
Then, for any $z,z' \in Z_T$, we have
\begin{equation}
    v_{\phi_{z}}\cdot v_{\phi_{z'}} = \sum_{j} \li \phi_z,\psi_j \ri\li \phi_{z'},\psi_j \ri = \li \phi_z,\phi_{z'} \ri.
\end{equation}
Now, we set
\begin{equation}
    G = \ld \alpha_{ij} \rd, \quad \alpha_{ij} = \sum_{z \in Z_T} \li f_z,\psi_i \ri\li f_z,\psi_j \ri.
\end{equation}
Now, we need to show that $G$ can be expressed as the linear combination of $v_{\phi_z}v_{\phi_z}^T$.
\begin{lem} \label{G}
    \begin{equation}
        G = \sum_{l =1}^{n+r} m_{x_l} v_{f_{x_l}}v_{f_{x_l}}^{T} + 2\sum_{l=1}^{k}m_{z_l}\ld v_{g_{z_l}}v_{g_{z_l}}^{T} - v_{h_{z_l}}v_{h_{z_l}}^{T} \rd,
    \end{equation}
    \begin{equation}
        tr(G) = N(T),
    \end{equation}
    and
    \begin{equation}
        \norm{G}^2_F = \sum_{z,s \in Z_T} K_{\delta}^2(z-s),
    \end{equation}
    where $\norm{\cdot}_F$ is the Frobenius norm.
\end{lem}
\begin{proof}
Split the sum in the definition of $\alpha_{ij}$ by the multiset $R_1,R_2,R_3$. Then
\begin{equation}
    \alpha_{ij} = \sum_{l=1}^{n+r} m_{x_l} \li f_{x_l},\psi_i \ri\li f_{x_l},\psi_j \ri + \sum_{l=1}^{k} m_{z_l} \ld \li f_{z_l},\psi_i \ri\li f_{z_l},\psi_j \ri + \li f_{\overline{z}_l},\psi_i \ri\li f_{\overline{z}_l},\psi_j \ri\rd.
\end{equation}
For the second sum, we have
\begin{equation}
\begin{aligned}
    \sum_{l=1}^{k} & m_{z_l} \ld \li f_{z_l},\psi_i \ri\li f_{z_l},\psi_j \ri + \li f_{\overline{z}_l},\psi_i \ri\li f_{\overline{z}_l},\psi_j \ri\rd \\
    &=\sum_{l=1}^{k} m_{z_l} \ld \li g_{z_l}+ih_{z_l},\psi_i \ri\li g_{z_l}+ih_{z_l},\psi_j \ri + \li g_{\overline{z}_l}+ih_{\overline{z}_l},\psi_i \ri\li g_{\overline{z}_l}+ih_{\overline{z}_l},\psi_j \ri\rd\\
    &=\sum_{l=1}^{k} m_{z_l} [ \li g_{z_l},\psi_i \ri\li g_{z_l},\psi_j \ri - \li h_{z_l},\psi_i \ri\li h_{z_l},\psi_j \ri +i\li h_{z_l},\psi_i \ri\li g_{z_l},\psi_j \ri + i\li g_{z_l},\psi_i \ri\li h_{z_l},\psi_j \ri\\
    &\quad+ \li g_{\overline{z}_l},\psi_i \ri\li g_{\overline{z}_l},\psi_j \ri - \li h_{\overline{z}_l},\psi_i \ri\li h_{\overline{z}_l},\psi_j \ri + i\li h_{\overline{z}_l},\psi_i \ri\li g_{\overline{z}_l},\psi_j \ri + i\li g_{\overline{z}_l},\psi_i \ri\li h_{\overline{z}_l},\psi_j \ri ].
\end{aligned}
\end{equation}
Observe that
\begin{equation}
    g_{\overline{z}}(u) = \frac{f_{\overline{z}}(u)+f_z(u)}{2} = g_z(u) \quad \text{and} \quad h_{\overline{z}}(u) = \frac{f_{\overline{z}}(u)-f_z(u)}{2i} = -h_z(u).
\end{equation}
Then, we obtain
\begin{equation}
\begin{aligned}
    \sum_{l=1}^{k} & m_{z_l} \ld \li f_{z_l},\psi_i \ri\li f_{z_l},\psi_j \ri + \li f_{\overline{z}_l},\psi_i \ri\li f_{\overline{z}_l},\psi_j \ri\rd \\
    &=2\sum_{l=1}^k m_{z_l} \ld \li g_{z_l},\psi_i \ri\li g_{z_l},\psi_j \ri - \li h_{z_l},\psi_i \ri\li h_{z_l},\psi_j \ri \rd
\end{aligned}
\end{equation}
This implies that we have
\begin{equation}
    \alpha_{ij} = \sum_{l=1}^{n+r} m_{x_l} \li f_{x_l},\psi_i \ri\li f_{x_l},\psi_j \ri + 2\sum_{l=1}^k m_{z_l} \ld \li g_{z_l},\psi_i \ri\li g_{z_l},\psi_j \ri - \li h_{z_l},\psi_i \ri\li h_{z_l},\psi_j \ri \rd,
\end{equation}
and
\begin{equation}
    G = \sum_{l =1}^{n+r} m_{x_l} v_{f_{x_l}}v_{f_{x_l}}^{T} + 2\sum_{l=1}^{k}m_{z_l}\ld v_{g_{z_l}}v_{g_{z_l}}^{T} - v_{h_{z_l}}v_{h_{z_l}}^{T} \rd.
\end{equation}
For $tr(G)$, we have
\begin{equation}
\begin{aligned}
    tr(G) = \sum_{j} \alpha_{jj} &= \sum_j \sum_{z\in Z_T}\li f_z,\psi_j \ri^2 \\
    &= \sum_j \ld \sum_{l=1}^{n+r} m_{x_l}\li f_{x_l},\psi_j \ri^2 + 2\sum_{l=1}^{k}m_{z_l}\lc \li g_{z_l},\psi_j \ri^2 - \li h_{z_l},\psi_j \ri^2 \rc  \rd\\
    &= \sum_{l=1}^{n+r} m_{x_l}\li f_{x_l},f_{x_l} \ri + 2\sum_{l=1}^{k}m_{z_l}\lc \li g_{z_l},g_{z_l} \ri - \li h_{z_l},h_{z_l} \ri \rc.
\end{aligned}
\end{equation}
Then, by \eqref{one1} and \eqref{one2}, we have
\begin{equation}
\begin{aligned}
    tr(G) &= \sum_{l=1}^{n+r} m_{x_l}\li f_{x_l},f_{x_l} \ri + 2\sum_{l=1}^{k}m_{z_l}\lc \li g_{z_l},g_{z_l} \ri - \li h_{z_l},h_{z_l} \ri \rc\\
    &= \sum_{l=1}^{n+r} m_{x_l} + 2\sum_{l=1}^{k}m_{z_l}\\
    &= N(T).
\end{aligned}
\end{equation}
Finally, we have
\begin{equation}
\begin{aligned}
    \norm{G}^2_F = \sum_{i,j} \alpha_{ij}^2 &= \sum_{i,j} \sum_{z,s \in Z_T} \li f_z,\psi_i \ri\li f_z,\psi_j \ri \li f_s,\psi_i \ri\li f_s,\psi_j \ri \\
    &= \sum_{z,s \in Z_T}\lc\sum_{i} \li f_z,\psi_i \ri\li f_s,\psi_i \ri\rc^2\\
    &= \sum_{z,s \in Z_T} \li f_z,f_{\overline{s}} \ri^2 \\
    &= \sum_{z,s \in Z_T} K_{\delta}^2(z-s).
\end{aligned}
\end{equation}
\end{proof}
For $x \in R_1$, we define
\begin{equation}
    P = \sum_{l=1}^n v_{f_{x_l}}v_{f_{x_l}}^T
\end{equation}
and set
\begin{equation}
    Q = G-P.
\end{equation}
By Lemma \ref{G}, we know that $G$ is Hermitian. Thus, we have $P,Q$ both Hermitian matrices. Moreover, one can observe that $P \succeq 0$ by its definition. Next, from the proof of lemma \ref{G}, we have
\begin{equation}
\begin{aligned}
    tr(P) &= \sum_{l=1}^n \sum_j \li f_{x_l},\psi_j \ri^2\\
    &= \sum_{l=1}^n \li f_{x_l},f_{x_l} \ri\\
    &= n.
\end{aligned}
\end{equation}
Since $tr(G) = tr(P) + tr(Q)$, we have
\begin{equation}
    tr(Q) = N(T) - n.
\end{equation}
Finally, observe that
\begin{equation}
    Q = \sum_{l =n+1}^{n+r} m_{x_l} v_{f_{x_l}}v_{f_{x_l}}^{T} + 2\sum_{l=1}^{k}m_{z_l}\ld v_{g_{z_l}}v_{g_{z_l}}^{T} - v_{h_{z_l}}v_{h_{z_l}}^{T} \rd.
\end{equation}
Set $n_+(Q)$ be the number of positive eigenvalues of $Q$, $n_-(Q)$ be the number of negative eigenvalues of $Q$, and
\begin{equation}
    V = \ld v_{f_{x_{n+1}}} \dots v_{f_{x_{n+r}}} v_{g_{z_1}} \dots v_{g_{z_k}} v_{h_{z_1}} \dots v_{h_{z_k}} \rd.
\end{equation}
Since the functions defining the columns of \(V\) are linearly independent, their coordinate vectors with respect to the basis \(\{\psi_j\}\) are also linearly independent. Hence \(V\) has full column rank.
Then, we have
\begin{lem} $n_+(Q) = r+k \quad \text{and} \quad n_-(Q) = k.$ \end{lem}
\begin{proof}
    Observe that
    \begin{equation}
        Q = V \begin{bmatrix}
        m_{x_{n+1}} & 0       & 0       & 0        & 0      & \cdots & 0           & 0      & 0 \\
        0           & \ddots  & 0       & 0        & 0      & \cdots & 0           & 0      & 0 \\
        0           & 0       & m_{x_{n+r}} & 0    & 0      & \cdots & 0           & 0      & 0 \\
        0           & 0       & 0       & 2m_{z_1} & 0      & \cdots & 0           & 0      & 0 \\
        0           & 0       & 0       & 0        & \ddots & \ddots & \vdots      & \vdots & \vdots \\
        \vdots      & \vdots  & \vdots  & \vdots   & \ddots & 2m_{z_k} & 0          & 0      & 0 \\
        0           & 0       & 0       & 0        & \dots  & 0      & -2m_{z_1}   & 0      & 0 \\
        0           & 0       & 0       & 0        & \dots  & 0      & 0           & \ddots & 0 \\
        0           & 0       & 0       & 0        & \dots  & 0      & 0           & 0      & -2m_{z_k}
        \end{bmatrix} V^T.
    \end{equation}
    Since each $m_z>0$, Sylvester's law of inertia~\cite{Bhatia1997} gives
    \begin{equation}
        n_+(Q) = r+k \quad \text{and} \quad n_-(Q) = k.
    \end{equation}
\end{proof}
Set
\begin{equation}
    \Psi(x) = 
    \begin{cases}
        x^2, \quad &x \in\ld 0,2 \rd\\
        4x-4,  \quad &x> 2.
    \end{cases}
\end{equation}
Observe that this function is convex, nonnegative, and $\Psi(0)=0$. Then we have
\begin{lem}\label{PQ}
    For any $p,q \geq 0$, we have
    \begin{equation}
        (p-q)^2 \geq \Psi(p) - 4q.
    \end{equation}
\end{lem}
\begin{proof}
    If $p \leq 2$, then $\Psi(p)=p^2$ and
    \begin{equation}
        (p-q)^2 - (p^2 - 4q) = q^2-2pq+4q = q(q-2p+4) \geq 0.
    \end{equation}
    If $p > 2$, then $\Psi(p)=4p-4$ and
    \begin{equation}
        (p-q)^2 - (4p - 4 - 4q) = p^2 -2pq +q^2 -4p + 4 + 4q = (q-(p-2))^2 \geq 0.
    \end{equation}
\end{proof}

\begin{lem} \label{P+Q} $\norm{P+Q}^2_{F} \geq tr\Psi(P) + 4tr(Q) - 4n_+(Q).$\end{lem}
\begin{proof}
    Diagonalize $Q = \sum_{l} \lambda_l u_l u_l^T$ with $\lb u_l \rb$ be orthonormal and set
    \begin{equation}
        Q_+= \sum_{\lambda_l>0} \lambda_l u_l u_l^T \quad \text{and} \quad Q_-= \sum_{\lambda_l<0} (-\lambda_l) u_l u_l^T.
    \end{equation}
    Then, $Q = Q_+ - Q_-$, where $Q_{\pm}\succeq 0$ and $Q_+Q_-=0$. Thus, we obtain
    \begin{equation}
        \norm{P+Q}^2_F = \norm{P-Q_-}^2_F + 2tr\lc \lc P-Q_- \rc Q_+ \rc + \norm{Q_+}_F^2.
    \end{equation}
    Since $P,Q_+ \succeq 0$, we have $tr(PQ_+)\geq0$ and
    \begin{equation}
        \norm{P+Q}^2_F \geq \norm{P-Q_-}^2_F + \norm{Q_+}_F^2.
    \end{equation}
    Let $q_{1},...,q_{n_+(Q)}>0$ be the positive eigenvalues of $Q$, then
    \begin{equation}
        \norm{Q_+}_F^2 = \sum_{j} q_j^2 \geq \sum_{j} (4q_j - 4) = 4tr(Q_+) -4n_+(Q). 
    \end{equation}
    Let $p_{1}\geq...\geq p_{\dim(W)}\ge 0$ and $\widetilde{q}_{1}\geq...\geq \widetilde{q}_{\dim(W)}\ge 0$ be the eigenvalues of $P$ and $Q_-$, respectively, including zero eigenvalues. Then
    \begin{equation}
        \norm{P-Q_-}^2_F = tr(P^2) + tr(Q_-^2)-2tr(PQ_-).
    \end{equation}
    Since von Neumann's trace inequality~\cite{Bhatia1997} gives $tr(PQ_-) \leq \sum_j p_j \widetilde{q}_j$, we obtain
    \begin{equation}
        \norm{P-Q_-}^2_F \geq tr(P^2) + tr(Q_-^2) - 2\sum_j p_j \widetilde{q}_j = \sum_{j} (p_j - \widetilde{q}_j)^2.
    \end{equation}
    By lemma \ref{PQ}, we have
    \begin{equation}
        \sum_{j} (p_j - \widetilde{q}_j)^2 \geq \sum_j (\Psi(p_j)-4\widetilde{q}_j) = \sum_j \Psi(p_j)-4tr(Q_-) = tr\Psi(P)-4tr(Q_-).
    \end{equation}
    Finally, we have
    \begin{equation}
        \norm{P+Q}^2_F \geq 4tr(Q_+) -4n_+(Q) + tr\Psi(P)-4tr(Q_-) = tr\Psi(P) + 4tr(Q) - 4n_+(Q).
    \end{equation}
\end{proof}
Recall that \(\operatorname{tr}(Q)=N(T)-n\) counts, with multiplicity, the zeros that are not both simple and on the critical line. we have
\begin{lem}\label{PsiBound} $ tr\Psi(P) + 2tr(Q) \leq \norm{G}_F^2$ \end{lem}
\begin{proof}
    By lemma \ref{P+Q}, we have
    \begin{equation}
    \begin{aligned}
        \norm{G}_F^2 \geq tr\Psi(P) + 4tr(Q) - 4n_+(Q)
    \end{aligned}
    \end{equation}
    Observe that $tr(Q) \geq 2(r+k)$ and $n_+(Q) = r+k$, then
    \begin{equation}
        \norm{G}_F^2 \geq tr\Psi(P) + 4tr(Q) -4(r+k)\geq tr\Psi(P) + 2tr(Q).
    \end{equation}
\end{proof}

\section{Lower Bound of $K_{\delta}(c)$}
\begin{lem} \label{lem:Kbound}
The function \(K_\delta\) is real-valued and even on \(\mathbb R\), and strictly decreasing on \((0,1)\).
\end{lem}
\begin{proof}
    By definition, we have
    \begin{equation}
    \begin{aligned}
        K_{\delta}(\xi) &= \hI \eta^2_{\delta}(u)e^{-2\pi i \xi u}du \\
        & = \hI \eta^2_{\delta}(u)\cos{(2\pi \xi u)}du - i\hI \eta^2_{\delta}(u)\sin{(2\pi \xi u)}du.
    \end{aligned}
    \end{equation}
    Since $\eta^2_{\delta}(u)$ is even and $\sin{(2\pi \xi u)}$ is odd, we have
    \begin{equation}
        \hI \eta^2_{\delta}(u)\sin{(2\pi \xi u)}du = 0.
    \end{equation}
    Thus, we must have
    \begin{equation}
        K_{\delta}(\xi) = \hI \eta^2_{\delta}(u)\cos{(2\pi \xi u)}du
    \end{equation}
    which implies $K_{\delta}(\xi)$ is real. Since $\cos{(2\pi \xi u)} = \cos{(-2\pi \xi u)}$, we have
    \begin{equation}
        K_{\delta}(-\xi) = K_{\delta}(\xi).
    \end{equation}
    Observe that
    \begin{equation}
    \begin{aligned}
        K'_{\delta}(\xi) &= -2\pi\hI \eta^2_{\delta}(u)\sin{(2\pi \xi u)}u \text{ }du\\
        &= -4\pi\int_{0}^{\h} \eta^2_{\delta}(u)\sin{(2\pi \xi u)}u \text{ }du.
    \end{aligned}
    \end{equation}
    For $\xi \in \lc 0,1 \rc$ and $u \in \lc 0, \h \rc$, we have
    \begin{equation}
        0 < 2\pi \xi u < \pi.
    \end{equation}
    Thus, we must have $\sin{(2\pi \xi u)} > 0$ and
    \begin{equation}
        K'_{\delta}(\xi) = -4\pi\int_{0}^{\h} \eta^2_{\delta}(u)\sin{(2\pi \xi u)}u \text{ }du<0.
    \end{equation}
    Finally, for $0<\xi_1<\xi_2<1$ and by the mean value theorem, there exists $\theta \in (\xi_1,\xi_2)$ such that
    \begin{equation}
        K_{\delta}(\xi_2)-K_{\delta}(\xi_1) = K'_{\delta}(\theta)(\xi_2-\xi_1)<0,
    \end{equation}
    which implies $K_{\delta}$ is strictly decreasing.
\end{proof}
Thus, by Lemma ~\ref{lem:Kbound}, for any $x,y \in R_1$ with $\la x-y \ra \leq c<1$, we have
\begin{equation} \label{K_d(c)}
    v_{f_{x}}\cdot v_{f_{y}} = \li f_{x},f_{y} \ri = K_{\delta}(x-y)=K_{\delta}(\la x-y\ra)\geq K_{\delta}(c).
\end{equation}
Observe that by \eqref{eta}
\begin{equation}
    K_{\delta}(c) = \hI \eta^2_{\delta}(u)e^{-2\pi i c u}du = \hI \frac{\widetilde{\eta}_{\delta}^2(u)A(u)}{A_\delta}e^{-2\pi i c u}du
\end{equation}
and
\begin{equation}
    \hat{A}(c) = \hI A(u) e^{-2\pi i c u}du = \frac{1}{\sqrt{2}\sin{\lc \frac{1}{\sqrt{2}} \rc}}\lc \frac{\sin{\lc \frac{1}{\sqrt{2}} +\pi c \rc}}{\sqrt{2}+2\pi c} + \frac{\sin{\lc \frac{1}{\sqrt{2}} -\pi c \rc}}{\sqrt{2}-2\pi c} \rc
\end{equation}
The quotient at \(c = 1 /(\sqrt{2} \pi)\) is interpreted by continuity. Then, we have
\begin{lem}
    For every fixed $c \in (0,1)$, there exists $\delta \in \lc 0, \frac{1}{4} \rc$ such that
    \begin{equation}
        K_{\delta}(c) \geq \hat{A}(c) - \frac{2\lc 1-A_{\delta} \rc}{A_\delta} > 0.
    \end{equation}
\end{lem}
\begin{proof}
First, we have
\begin{equation}
\begin{aligned}
    \la K_{\delta}(c) - \hat{A}(c) \ra &\leq \hI \la \frac{\widetilde{\eta}_{\delta}^2(u)A(u)}{A_\delta}-A(u) \ra du\\
    &= \frac{1}{A_\delta}\hI \la A(u) \ra \la \widetilde{\eta}_{\delta}^2(u)-A_\delta \ra du\\
    &\leq \frac{1}{A_\delta}\hI \la A(u) \ra \lc \la 1- \widetilde{\eta}_{\delta}^2(u) \ra + \la1 -A_\delta \ra \rc du.
\end{aligned}
\end{equation}
Observe that 
\begin{equation}
    A>0, \quad0 < A_{\delta} <1, \quad \text{and} \quad \widetilde{\eta}_{\delta}^2 \leq 1.
\end{equation}
Then, we have
\begin{equation}
\begin{aligned}
    \la K_{\delta}(c) - \hat{A}(c) \ra & \leq \frac{1}{A_\delta}\hI A(u) \lc 1- \widetilde{\eta}_{\delta}^2(u) + 1 -A_\delta \rc du \\
    &= \frac{2\lc 1-A_{\delta} \rc}{A_\delta}.
\end{aligned}
\end{equation}
Then we obtain
\begin{equation}
    K_{\delta}(c) \geq \hat{A}(c) - \frac{2\lc 1-A_{\delta} \rc}{A_\delta}.
\end{equation}
Since $\frac{2\lc 1-A_{\delta} \rc}{A_\delta} \xrightarrow{} 0$ as $\delta \xrightarrow{}0$ and $\hat{A}(c) > 0$ for all $c \in (0,1)$, we can always pick $\delta \in \lc 0,\frac{1}{4} \rc$ that is sufficiently small such that 
\begin{equation}
    \hat{A}(c) - \frac{2\lc 1-A_{\delta} \rc}{A_\delta} >0
\end{equation}
when $c$ is fixed.
\end{proof}
We can now obtain a suitable lower bound for \(K_\delta(c)\):
\begin{lem}
For \(0 < \delta \leq \frac{1}{200}\), we have \(A_\delta \geq \frac{99}{100}\).
Moreover,
\[
K_\delta(c)^2 > \frac34
\qquad \forall \ 0 < c \leq 1/4.
\]
\end{lem}
\begin{proof}
Since \(\widetilde\eta_\delta(u) = 1\) for \(|u| \leq \frac{1}{2} - \delta\), the definition of \(A_\delta\) gives
\[
\begin{aligned}
A_\delta
&\geq \int_{-\frac12 + \delta} ^{\frac12 - \delta} A(u) du\\
&= \frac{\sin((1 - 2\delta)/\sqrt2)}{\sin(1/\sqrt2)}\\
&\geq 1 - 2\delta\\
&\geq \frac{99}{100}.
\end{aligned}
\]

By lemma~\ref{lem:Kbound} , \(K_\delta(c)\) is strictly decreasing, we have
\[
K_\delta(c) \geq K_\delta(\frac14) \geq \hat{A}(1/4) - 2\frac{1 - A_\delta}{A_\delta} \geq \hat{A}(\frac14) - \frac{2}{99} > \frac{\sqrt3}{2}
\]

\end{proof}

\section{Spectral estimates for pairs, triples and blocks}

Let
\begin{equation}
    \widetilde{R_1} = \lb x \in R_1 : \exists x' \in R_1 \text{ s.t } 0<\la x-x' \ra <c \rb
\end{equation}
be the set of simple critical zeros that are ``too close" and write its elements in increasing order as
\begin{equation}
    x^{(1)} < x^{(2)}< \dots < x^{\lc \la \widetilde{R_1} \ra \rc}.
\end{equation}

Now, start a new block whenever the distance between successive elements is at least \(c\).
Denote the resulting blocks by \(B_i\), so that
\begin{equation}
\widetilde {R_1} = \bigsqcup_i B_i 
\end{equation}
Every block contains at least two elements. Indeed, if \(x\in \widetilde{R_1}\), then any \(x' \in R_1\) with \(0<|x-x'|<c\) also belongs to \(\widetilde{R_1}\).
Such points belong to the same block, so a singleton block cannot occur.
Moreover, each \(B_i\) is a consecutive segment of the ordered set \(R_1\). If an element of \(R_1\) lay strictly between two successive elements of a block, it would be within distance \(c\) of these elements and would itself belong to \(\widetilde{R_1}\).
Thus successive elements of \(B_i\) are also successive in \(R_1\), and their distances are less than \(c\).
Fix \(0< c \leq 1/4\), and choose \(\delta > 0 \) sufficiently small that the kernel estimates above give
\[0< K_\delta(c) \le 1,
\qquad
K_\delta(c)^2>\frac34.\]
Recall that for simple real elements \(x,y \in R_1\),
\[\|v_{f_x}\| = 1, 
\qquad
v_{f_x}\cdot v_{f_{y} }= K_\delta(x - y)\]
In particular, if \(|x - y|< c\), then
\[v_{f_x}\cdot v_{f_{y}} \ge K_\delta(c)\]
We first establish estimates for two and three points.
\begin{lem}\label{lem:pair}
For distinct simple real elements \(x,y\), one has
\[tr \Psi (v_{f_x}v_{f_{x}}^T + v_{f_{y}}v_{f_{y}}^T)
=
2 + 2 K_\delta(x-y)^2\]
Consequently, if \(|x-y|<c\), then 
\[tr \Psi (v_{f_x}v_{f_{x}}^T + v_{f_{y}}v_{f_{y}}^T)
\geq
2 + 2 K_\delta(c)^2\]
\end{lem}
\begin{proof}
By the standard nonzero-spectrum identity for Gram matrices~\cite{Bhatia1997},
the nonzero eigenvalues of the matrix in question coincide with those of the Gram matrix
\[
\begin{pmatrix}
1 & K_\delta(x -y)\\
K_\delta(x-y) & 1
\end{pmatrix}.
\]
They are \(1+ |K_\delta(x-y)|\) and \(1 - |K_\delta(x-y)|\). 
Both belong to \([0,2]\), where \(\Psi(t) = t^2\).
Since \(\Psi(0) = 0\), the trace is therefore 
\[
(1+ |K_\delta(x-y)|)^2 + (1 - |K_\delta(x-y)|)^2 = 2 + 2K_\delta(x-y)^2
\]
The inequality follows from the lower bound of the kernel.
\end{proof}

The next estimate will be used for blocks of odd size.
\begin{lem}\label{lem:triple}
Let \(x<y<z\) belong to \(R_1\), with
\[
y-x<c,
\qquad
z-y<c
\]
Then
\[
tr\Psi(v_{f_x}v_{f_{x}}^T + v_{f_{y}}v_{f_{y}}^T + v_{f_{z}}v_{f_{z}}^T) \geq 3 + 3 K_\delta(c)^2 
\]
\end{lem}
\begin{proof}
Let \(H\) be the Gram matrix of the three vectors.
The trace in the statement equals
\(tr\Psi(H)\).
Since \(H\) is positive semi-definite with diagonal entries equal to \(1\),
\[
trH = 3,
\qquad
\|H-I_3\|_F^2 = tr(H-I_3)^2 = tr(H^2) - 2tr(H) + tr(I_3) = tr(H^2) - 3.
\]
Moreover,since 
\[|x-y|,|y-z|\leq c \implies K_\delta(x-y)^2, K_\delta(y-z)^2 \ge K_\delta(c)^2\]
and
\[
K_\delta(x-z)^2 \geq 0.
\]
We have:
\[
\begin{aligned}
\|H-I_3\|_F^2
&=
2K_\delta(x-y)^2
+2K_\delta(y-z)^2
+2K_\delta(z-x)^2\\
&\geq 4K_\delta(c)^2.
\end{aligned}
\]

If all eigenvalues of \(H\) are at most \(2\), then
\[
tr\Psi(H)
=
tr(H^2)
\geq 3+4K_\delta(c)^2
\]
which implies the desired estimate.
Otherwise, let \(\mu>2\) be the largest eigenvalue, \(\lambda_1,\lambda_2\) be the other two eigenvalues. \(\lambda_1, \lambda_2\) are nonnegative and have sum \(3-\mu<1\), so they are less than \(2\).
Consequently,
\[
tr\Psi(H)
=
tr(H^2) - \mu^2 + 4\mu - 4 = \|H - I_3\|_F^2 +3 - (\mu-2)^2
\]
Since
\[
(\mu - 1)^2 \leq (\mu-1)^2 + (\lambda_1-1)^2 + (\lambda_2-1)^2 \leq \|H-I_3\|^2_F
\]
It follows that
\[
\begin{aligned}
tr\Psi(H) - 3
&\geq
\|H-I_3\|_F^2 - (\|H-I_3\|_F -1)^2\\
&=
2\|H-I_3\|_F - 1\\
&\geq
4K_\delta(c)-1\\
&\geq
3K_\delta(c)^2.
\end{aligned}
\]
The last inequality follows from
\(\frac13 \leq K_\delta(c) \leq 1\).
\end{proof}
Thus, each pair or triple contributes at least \(1+K_\delta(c)^2\) per point on average.
Since $\Psi$ is nondecreasing and convex on $[0,\infty)$ with
$\Psi(0)=0$, the convex trace superadditivity inequality gives
\[
\operatorname{tr}\Psi(A+B)
\geq \operatorname{tr}\Psi(A)+\operatorname{tr}\Psi(B)
\qquad (A,B\succeq0);
\]
see \cite[Corollary~3.2]{BourinLee2012}.
Applying this inequality to the positive semidefinite matrices
associated with the pairs and triples in each block gives the following lemma.

\begin{lem}\label{lem:block}
For every block \(B_i\),
\[
tr\Psi(\sum_{x\in B_i}v_{f_{x}}v_{f_{x}}^T)
\geq
(1 + K_\delta(c)^2)|B_i|.
\]
Consequently,
\[
tr\Psi(P)
\ge
n+K_\delta(c)^2|\widetilde{R_1}|,
\]
and
\begin{equation}
tr(Q) + K_\delta(c)^2|\widetilde{R_1}|
\leq
\|G\|_F^2 - N(T).
\end{equation}
\end{lem}
\begin{proof}
If \(|B_i|\) is even, partition its elements into successive pairs. If \(|B_i|\) is odd, take its first three elements as a triple and partition the remaining elements into successive pairs. Since \(|B_i| \geq 2\), this covers every element exactly once. Lemmas~\ref{lem:pair} and ~\ref{lem:triple}, together with the preceding trace superadditivity inequality, give
\[
tr\Psi(\sum_{x\in B_i}v_{f_{x}}v_{f_{x}}^T)
\geq
\sum_{x\in B_i} (1+K_\delta(c)^2)
=
(1+K_\delta(c)^2)|B_i|
\]
Then use the same trace superadditivity inequality for \(tr\Psi(P)\):

\begin{align}
tr\Psi(P)
&\geq
\sum_{x\in R_1-\widetilde{R_1}} tr\Psi(v_{f_{x}}v_{f_x}^T) + \sum_{B_i\ \text{blocks}} tr\Psi(\sum_{x\in B_i} v_{f_x}v^T_{f_x})\\
&\geq
|R_1 - \widetilde{R_1}| + |\widetilde{R_1}|(1+K_\delta(c)^2)\\
&=
n + K_\delta(c)^2|\widetilde{R_1}|
\end{align}
Finally, lemma ~\ref{PsiBound} gives
\[
tr\Psi(P)
+
2tr(Q)
\leq
\|G\|_F^2
\]
Substitute the preceding estimate and use \(n = N(T) - trQ\) to obtain:
\begin{equation}
tr(Q) + K_\delta(c)^2|\widetilde{R_1}|
\leq
\|G\|_F^2 - N(T).
\end{equation}
\end{proof}

\section{Clusters and Consecutive gaps}
Let $z \in Z_T$, we define multi-set
\begin{equation}
    \mathcal{C}_z = \lb s \in Z_T : \mathrm{Re}(s) = \mathrm{Re}(z) \rb.
\end{equation}
Set
\begin{equation}
    \text{Simple case: }Z_T^s= \lb \mathcal{C}_z : |\mathcal{C}_z|=1 \rb,
\end{equation}
\begin{equation}
        \text{Mixed case: }Z_T^m= \lb \mathcal{C}_z : \exists x \in \mathcal{C}_z \text{ s.t } x \in R_1 \text{ and } |\mathcal{C}_z|>1 \rb,
\end{equation}
\begin{equation}
    \text{No simple case: }Z_T^n= \lb \mathcal{C}_z : \nexists x \in \mathcal{C}_z \text{ s.t } x \in R_1 \rb,
\end{equation}
Each of these three collections consists of distinct clusters.
Denote $\la Z_T^s \ra$, $\la Z_T^m \ra$, and $\la Z_T^n \ra$ as the number of $\mathcal{C}_z$ in each set. One can check that
\[
Z_T
=
\bigsqcup_{\mathcal C\in Z_T^s\cup Z_T^m\cup Z_T^n}
\mathcal C.
\]

Since
\[
tr(Q) = N(T) - n,
\qquad
\forall z\in R_2 \cup R_3, |\mathcal{C}_z| \geq 2 
\]
Then
\[
|Z_T^m| + |Z_T^n| \leq \frac{tr(Q)}{2} . 
\]
\begin{prop}
For every integer \(\ell \geq 2\),
\[
w_\ell(T,c) \geq N(T) - \frac{\ell + 2}{2} tr(Q) - \frac{\ell}{2}  |\widetilde{R_1}|- \ell
\]
\end{prop}
\begin{proof}
Order the clusters by ordinate and list the elements of \(R_1\) in the corresponding order.

We first begin with a string of \(\ell +1\) consecutive elements  of \(R_1\). There are at least \(n - \ell \) such strings.
Discard a string if one of its interior elements belongs to a cluster in \(Z_T^m\), or if a cluster in \(Z_T^n\) lies between its first and last ordinates. Each cluster in \(Z_T^m\) invalidates at most \(\ell - 1\) strings, Each cluster in \(Z_T^n\) lies between at most one pair of consecutive elements of \(R_1\), and hence invalidates at most \(\ell\) strings.
This leaves at least\[
n - \ell - (\ell - 1)|Z_T^m| - (\ell)|Z_T^n|
\]
candidates.

We now discard candidates containing a short gap. Since if there is a short gap between point \(x,y \in R_1\), then there must exist a block \(B_i\) such that \(x,y \in B_i\). Then each block \(B_i\) can invalidate at most \(|B_i| + \ell -2\) candidates. 
Since \(|B_i|,\ell \geq 2\),
\[
|B_i| + \ell - 2 \leq \frac{\ell }{2} |B_i|.
\]
The assumption \(\ell\ge2\) is used at this step, since the above inequality does not hold in general for \(\ell = 1\). We have also obtained quantitative bounds for the single-gap case by a separate argument; these results are beyond the scope of the present paper.

Summing over the blocks, the total loss is at most
\[
\frac{\ell}{2} |\widetilde{R_1}|
\]

Consequently,
\[
\begin{aligned}
w_\ell(T,c)
&\geq n - \ell - (\ell - 1)|Z_T^m| - (\ell)|Z_T^n| - \frac{\ell}{2} |\widetilde{R_1}|\\
&\geq n - \ell - \ell (|Z_T^m| + |Z_T^n|) - \frac{\ell}{2} |\widetilde{R_1}|\\
&\geq n - \frac{\ell}{2} tr(Q) - \frac{\ell}{2} |\widetilde{R_1}| - \ell\\
&\geq N(T) - \frac{\ell+2}{2} tr(Q) - \frac{\ell}{2} |\widetilde{R_1}| - \ell
\end{aligned}
\]

\end{proof}

\section{Proof of Theorem 1}

Fix \(0 < c \leq \frac{1}{4}\).
The spectral estimate of Section~6 gives
\[
tr(Q) + K_\delta(c)^2|\widetilde{R_1}|
\leq \|G\|_F^2 - N(T).
\]
Now assume \(2 \leq \ell \leq 4\).
By the choice of \(\delta\) in Section~5,
\[
\frac{\ell}{\ell+2} 
\leq \frac{3}{4}
< K_\delta(c)^2
\]
It follows that
\[
\begin{aligned}
\frac{\ell + 2}{2} tr(Q) + \frac{\ell}{2} |\widetilde{R_1}|
&\leq \frac{\ell+2}{2} (tr(Q) + K_\delta(c)^2|\widetilde{R_1}|)\\
&\leq \frac{\ell + 2}{2} (\|G\|_F^2 - N(T)).
\end{aligned}
\]
Combining this with the preceding proposition gives
\[
w_\ell(T,c) 
\geq N(T) - \frac{\ell + 2}{2} (\|G\|_F^2 - N(T)) - \ell.
\]
Keeping \(\delta\) fixed, Lemma~\ref{Youness} yields
\[
\|G\|_F^2
=
\lc C_{\eta_{\delta}}+O_{\delta}\lc \frac{1}{\sqrt{\log T}}\rc \rc N(T)
\]
Hence
\[
\frac{w_{\ell}(T,c)}{N(T)}
\geq
1 - \frac{\ell + 2}{2} \lc C_{\eta_{\delta}} - 1 +O_{\delta}\lc \frac{1}{\sqrt{\log T}}\rc \rc  - \frac{\ell}{N(T)}
\]
First, let \(T \to \infty\), with \(\delta\) fixed. Then let \(\delta \to 0\), lemma~\ref{Youness} shows that \(C_{\eta_{\delta}} \to C_{\mathrm{MT}}\).
We obtain, for every fixed \(0 < c \leq \frac{1}{4}\) and \(2\leq \ell \leq 4\),
\[
\boxed{\liminf_{T \to \infty} \frac{w_\ell(T,c)}{N(T)}}
\geq 1 - \frac{\ell+2}{2}(C_{\mathrm{MT}} -1)
= \frac{\ell + 4 - (\ell + 2)C_{\mathrm{MT}}}{2}.
\]
In particular, for each fixed $0<c\leq 1/4$, we obtain
\[
\liminf_{T\to\infty}\frac{w_\ell(T,c)}{N(T)}
\geq
\begin{cases}
0.3450014, & \ell=2,\\
0.1812517, & \ell=3,\\
0.0175021, & \ell=4.
\end{cases}
\]

\section*{AI assistance}
ChatGPT (OpenAI; GPT-5.6 Sol and GPT-6 Astra) was used to explore
and develop parts of the mathematical argument, assist with
intermediate calculations, and help revise the manuscript.
The authors take responsibility for the mathematical content
and the final manuscript.

\section*{Acknowledgments}

We would like to thank Alexandra Florea and Nathan Kaplan, the
organizers of ANTeater 2026 at UC Irvine, for creating a stimulating
environment for the exchange of ideas and for bringing together a
broad community interested in analytic number theory. We are grateful
to Anurag Sahay for suggesting this problem during ANTeater 2026.

We also thank Tianyu Zhao, Somin Kim, and Kai Hubbard for the helpful discussions of this problem in our working group at ANTeater 2026. We are particularly grateful to Tianyu Zhao for pointing out that Montgomery's pair correlation can be used to obtain positive-proportion results for consecutive moderate gaps under the Riemann hypothesis, without assuming the pair correlation conjecture, and for suggesting that the argument of Alpöge and Furman might lead to unconditional results.

We are deeply grateful to Erick Ross for his contributions to this project, for his passion for mathematics, and for his remarkable
ability to organize ideas and see the deeper structure behind them.
Erick constantly encouraged us not to stop at the first interesting
result, but to look further for ideas and questions with broader and
more lasting reach. We will remember just as strongly his kindness,
generosity, and extraordinary willingness to work with others.

\bibliographystyle{plainnat}
\bibliography{gap_2_to_4_references}

@misc{GS24,
  author        = {Gonek, Steven M. and Sahay, Anurag},
  title         = {Consecutive moderate gaps between zeros of the {Riemann} zeta function},
  year          = {2024},
  howpublished  = {arXiv preprint arXiv:2412.15481},
  eprint        = {2412.15481},
  archivePrefix = {arXiv},
  primaryClass  = {math.NT},
  doi           = {10.48550/arXiv.2412.15481},
  url           = {https://arxiv.org/abs/2412.15481}
}

@misc{AF26,
  author        = {Alp{\"o}ge, Levent and Furman, Ralph},
  title         = {More than two thirds of the zeta zeros are simple and on the critical line},
  year          = {2026},
  howpublished  = {arXiv preprint arXiv:2608.13637},
  eprint        = {2608.13637},
  archivePrefix = {arXiv},
  primaryClass  = {math.NT},
  doi           = {10.48550/arXiv.2608.13637},
  url           = {https://arxiv.org/abs/2608.13637v2},
  note          = {Version 2}
}

@misc{Lam26,
  author        = {Lamzouri, Youness},
  title         = {A new proof that more than {$2/3$} of the zeros of the {Riemann} zeta function are simple and on the critical line},
  year          = {2026},
  howpublished  = {arXiv preprint arXiv:2609.02882},
  eprint        = {2609.02882},
  archivePrefix = {arXiv},
  primaryClass  = {math.NT},
  doi           = {10.48550/arXiv.2609.02882},
  url           = {https://arxiv.org/abs/2609.02882v2},
  note          = {Version 2}
}

@article{BGSTB24,
  author        = {Baluyot, Siegfred Alan C. and Goldston, Daniel Alan
                   and Suriajaya, Ade Irma and Turnage-Butterbaugh, Caroline L.},
  title         = {An unconditional {Montgomery} theorem for pair correlation of zeros of the {Riemann} zeta-function},
  journal       = {Acta Arithmetica},
  volume        = {214},
  year          = {2024},
  pages         = {357--376},
  doi           = {10.4064/aa230612-20-3},
  eprint        = {2306.04799},
  archivePrefix = {arXiv},
  primaryClass  = {math.NT},
  url           = {https://doi.org/10.4064/aa230612-20-3}
}

@article{BourinLee2012,
  author        = {Bourin, Jean-Christophe and Lee, Eun-Young},
  title         = {Unitary orbits of {Hermitian} operators with convex or concave functions},
  journal       = {Bulletin of the London Mathematical Society},
  volume        = {44},
  number        = {6},
  pages         = {1085--1102},
  year          = {2012},
  doi           = {10.1112/blms/bds080},
  eprint        = {1109.2384},
  archivePrefix = {arXiv},
  primaryClass  = {math.FA},
  url           = {https://doi.org/10.1112/blms/bds080}
}

@book{Bhatia1997,
  author    = {Bhatia, Rajendra},
  title     = {Matrix Analysis},
  series    = {Graduate Texts in Mathematics},
  volume    = {169},
  publisher = {Springer},
  address   = {New York},
  year      = {1997},
  doi       = {10.1007/978-1-4612-0653-8},
  isbn      = {978-0-387-94846-1},
  url       = {https://doi.org/10.1007/978-1-4612-0653-8}
}

@book{Titchmarsh1986,
  author    = {Titchmarsh, E. C.},
  title     = {The Theory of the {Riemann} Zeta-Function},
  edition   = {2},
  publisher = {The Clarendon Press, Oxford University Press},
  address   = {Oxford},
  year      = {1986},
  isbn      = {0-19-853369-1},
  note      = {Revised by D. R. Heath-Brown}
}

@incollection{Mon75,
  author    = {Montgomery, Hugh L.},
  title     = {Distribution of the zeros of the {Riemann} zeta function},
  booktitle = {Proceedings of the International Congress of Mathematicians
               (Vancouver, 1974)},
  volume    = {1},
  publisher = {Canadian Mathematical Congress},
  address   = {Montreal},
  year      = {1975},
  pages     = {379--381}
}

@article{CCLM17,
  author        = {Carneiro, Emanuel and Chandee, Vorrapan
                   and Littmann, Friedrich and Milinovich, Micah B.},
  title         = {Hilbert spaces and the pair correlation of zeros of the {Riemann} zeta-function},
  journal       = {Journal f{\"u}r die reine und angewandte Mathematik},
  volume        = {725},
  year          = {2017},
  pages         = {143--182},
  doi           = {10.1515/crelle-2014-0078},
  eprint        = {1406.5462},
  archivePrefix = {arXiv},
  primaryClass  = {math.NT},
  url           = {https://doi.org/10.1515/crelle-2014-0078}
}

@misc{Lam26Isolated,
  author        = {Lamzouri, Youness},
  title         = {Isolated zeros of the {Riemann} zeta function and
                   zeros of linear combinations of {Dirichlet}
                   {$L$}-functions},
  year          = {2026},
  eprint        = {2610.26841},
  archivePrefix = {arXiv},
  primaryClass  = {math.NT},
  doi           = {10.48550/arXiv.2610.26841},
  url           = {https://arxiv.org/abs/2610.26841}
}

\end{document}